\documentclass[11pt,twoside,letter]{article} 

\usepackage{amssymb}
\usepackage{graphicx,rotating}
\usepackage{amscd}
\usepackage{color}
\usepackage{float}
\usepackage{graphics}
\usepackage{amsmath}
\usepackage{amssymb}
\usepackage{mathrsfs}
\usepackage{amsthm}
\usepackage{amsfonts}
\usepackage{latexsym}
\usepackage{epsf}
\usepackage[]{hyperref}
\usepackage{fancyhdr}
\usepackage{mathptmx}
\usepackage{longtable}
\usepackage{array}
\usepackage{csquotes}
\usepackage{verbatim}
\usepackage{enumerate}
\usepackage[margin=1.5in]{geometry}
\usepackage{rotating}
\usepackage{caption}
\usepackage{mdframed}
\usepackage{colortbl}

\newcolumntype{P}[1]{>{\centering\arraybackslash}p{#1}}
\newmdtheoremenv{theo}{Theorem}

\makeatletter
\newenvironment{proof*}[1][\proofname]{\par
  \pushQED{\qed}%
  \normalfont \partopsep=\z@skip \topsep=\z@skip
  \trivlist
  \item[\hskip\labelsep
        \itshape
    #1\@addpunct{.}]\ignorespaces
}{%
  \popQED\endtrivlist\@endpefalse
}
\makeatother

\date{}

\begin{document}



\centerline {\Large{\bf Distance Eigenvalues and Forwarding Indices of Multiplicative}} 

\centerline{\Large{\bf Circulant Graph of Order Power of Two and Three}} 


\vspace{2mm}










\newtheorem{theorem}{Theorem}[section]
\newtheorem{lemma}[theorem]{Lemma}
\newtheorem{corollary}[theorem]{Corollary}
\newtheorem{proposition}[theorem]{Proposition}
\newtheorem{definition}[theorem]{Definition} 
\newtheorem{example}[theorem]{Example}
\newtheorem{remark}[theorem]{Remark}
\newtheorem{illustration}[theorem]{Illustration}


\begin{abstract} 
In this paper, we use Breadth-first search algorithm to determine the distance matrix of multiplicative circulant graph of order power of two and three. As a consequence, the diameter of the graphs were determined. We also give their distance spectral radii, average distances, as well as the exact values of vertex-forwarding indices. Finally, using some known relationships between the distance spectral radii and forwarding indices of a graph, we give some bounds for their edge-forwarding indices.
\end{abstract} 
{\small{{\bf Keywords:} Breadth-first Search Algorithm, Multiplicative Circulant Graph, Graph Distance Matrix, Graph Diameter, Graph Distance Spectral Radius, Graph Average Distance, Graph Edge and Vertex-Forwarding Indices\\
   {\bf AMS Classification Numbers:} 05C12, 05C50, 05C85}}
\hrule
\vspace{4mm}

\begin{flushleft}
{\footnotesize{{\bf Author Information:}
\bigskip

\underline{John Rafael M. Antalan} 
\bigskip

(Faculty, Department of Mathematics and Physics, College of Science, Central Luzon State University, 3120), Science City of Mu\~{n}oz, Nueva Ecija, Philippines.
\bigskip

(Graduate Student, Mathematics and Statistics Department, College of Science, De La Salle University, 00000), 2401 Taft Avenue, Malate, Manila, 1004 Metro Manila, Philippines
\bigskip

e-mail: jrantalan@clsu.edu.ph  
\bigskip 

\underline{Francis Joseph H. Campe\~{n}a} 
\bigskip

(Faculty, Mathematics and Statistics Department, College of Science, De La Salle University, 00000), 2401 Taft Avenue, Malate, Manila, 1004 Metro Manila, Philippines
\bigskip

e-mail: francis.campena@dlsu.edu.ph   

}}

\end{flushleft}

\newpage

\section{Introduction}

Unless explicitly stated, the graphs considered in this paper are all simple and connected. Moreover, the notations of Bondy and Murty \cite{Bondy} as well as of Fraleigh \cite{Fraleigh} were adapted for notations not explicitly defined in this paper.  

Let $G$ be a group and $S$ be a subset of $G\backslash\{e\}$. A graph $\Gamma$ is a {\it{\bf Cayley graph}} of $G$ with connection (or jump) set $S$, written $\Gamma=Cay(G,S)$ if  $V(\Gamma)=G$ and 
$E(\Gamma)=\{\{g,sg\}:g\in G, s\in S\}$. 

If $G=\langle\mathbb{Z}_n,+_n\rangle$, then the graph $\Gamma=Cay(G,S)$ is called a {\it{\bf circulant graph}} with connection set $S$. Note that for $s$ and $s^{-1}$ in $\mathbb{Z}_n$, we have $\{\{g,s+_n g\}:g\in G\}=\{\{g,s^{-1}+_n g\}:g\in G\}$. Hence, for a circulant graph, we have $S\subseteq \{1,2,\ldots,\lfloor \frac{n+1}{2} \rfloor\}$. Moreover, a circulant graph has a circulant adjacency matrix. Figure \ref{imcirc1} shows some examples of circulant graphs.           

\begin{figure}[h!]
\begin{center}
\includegraphics[width=0.25\columnwidth]{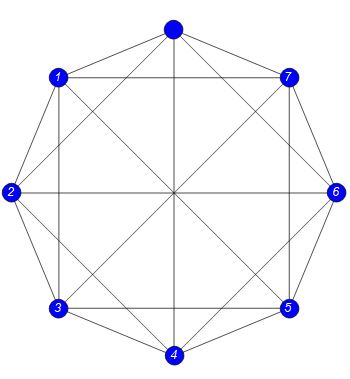}
\hspace{5mm}
\includegraphics[width=0.25\columnwidth]{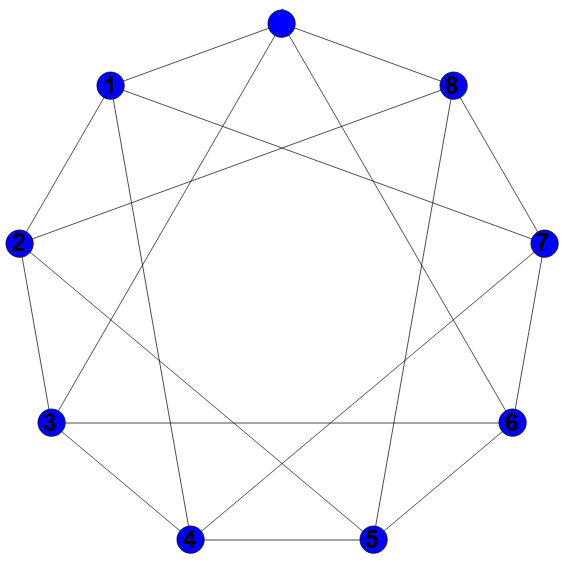}
\hspace{5mm}
\includegraphics[width=0.25\columnwidth]{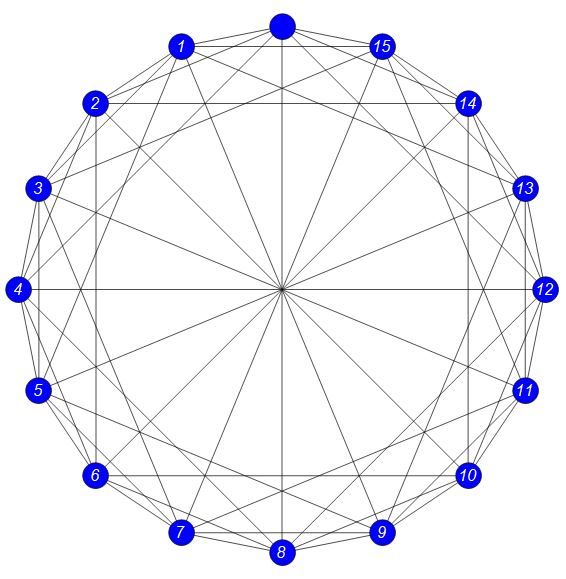}
\caption{{ The graphs $Cay(\mathbb{Z}_8,\{1,2,4\})$, $Cay(\mathbb{Z}_9,\{1,3\})$ and $Cay(\mathbb{Z}_{16},\{1,2,4,8\})$.
{\label{imcirc1}}
}}
\end{center}
\end{figure}

Circulant graphs have vast applications in different fields of study; some of these fields include telecommunication networking \cite{Bermond}, VLSI (Very-large-scale integration) design \cite{Leighton}, and distributed computing \cite{Mans}.    

This paper is motivated by the work of Liu, Lin and Shu \cite{Liu}. Liu et al. determined the first row of the distance matrix, distance spectral radius and vertex-forwarding index of the circulant  graphs $Cay(\mathbb{Z}_n,S)$ where $S$ is either $\{1,d\}$ or $\{1,2,\ldots,d\}$ and $2\leq d=\frac{n}{2}$. They also obtained an upper and lower bound for the edge-forwarding index of the said graphs. This study extends their results to multiplicative circulant graphs of order power of two and three.    

The rest of the paper is organized as follows: In section 2 (Important Definitions and Preliminaries), we formally define multiplicative circulant graphs and state some of their basic properties. Also in section 2, we discuss the concept of graph distance matrix, distance spectral radius, average distance, as well as the Breadth-First Search Algorithm, an important algorithm in finding the distance matrix of any graph. Section 2 ends with the discussion of graph vertex-forwarding index and graph edge-forwarding index. The determination of the distance matrix, distance spectral radius, average distance, vertex-forwarding index, and the upper and lower bound for the edge-forwarding index of multiplicative circulant graph of order power of two and three is presented in section 3 (Results). Finally in section 4 (Conclusion and Future Work) we recommend some possible extensions and generalizations for this research work.

\section{Important Definitions and Preliminary Results} 
In this section, we formally define multiplicative circulant graphs and state some of their basic properties. A discussion of the breath-first search method and an illustration of the method on multiplicative circulant graphs are also included which will be used to obtain results on the distance spectral radii, diameters, average distances, vertex and edge forwarding indices of the graphs in Section 3.

\subsection{Multiplicative Circulant Graphs}\label{circclass}

The graphs $Cay(\mathbb{Z}_8,\{1,2,4\})$, $Cay(\mathbb{Z}_9,\{1,3\})$, and $Cay(\mathbb{Z}_{16},\{1,2,4,8\})$ in Figure \ref{imcirc1} are some examples of multiplicative circulant graphs. We formally define multiplicative circulant graphs in our first definition. 

\begin{definition}
Let $m>1$ and $h>0$ be integers. A {\bf multiplicative circulant graph} or $MC$ graph is a circulant graph with $m^h$ vertices and connection set $S=\{m^0,m^1,m^2,\ldots,m^{h-1}\}$.
\end{definition}

We denote multiplicative circulant graph on $m^h$ vertices by $MC(m^h)$. The name ``multiplicative circulant" was given by Stojmenovic \cite{Stoj} in 1997 when he studied a particular class of circulant graph introduced by Park and Chwa \cite{Park1} called {\it recursive circulant graph}. 

\begin{definition}
A circulant graph $Cay(\mathbb{Z}_n,S)$ is called a {\bf recursive circulant graph} if $n$ can be expressed as $e\cdot m^h$ where $e<m$, $k=\lceil\mbox{log}_mn\rceil$, $s_i=m^{i-1}$ for $i=\{1,2,\ldots,k\}$. 
\end{definition}

\begin{remark}
MC graphs are recursive circulant graphs with $e=1$. A larger class of circulant graphs which contains all recursive circulant graphs and hence contains all $MC$ graphs was introduced by Tang et al. \cite{Tang} in 2012 called {\it generalized-recursive circulant graph} or GRC graph denoted by $GR(m_h,m_{h-1},\ldots,m_1)$. The vertices in a GRC graph are expressed as $h-tuple$ $(x_h,x_{h-1},\ldots,x_1)$ where for each $i=1,2,\ldots, h$, $0\leq x_i<m_i$. The number $i$ denotes the dimension of the labeling system while $m_i$ refers to the base of the dimension $i$. A dimension $i$ is said to be even if $m_i$ is even. Lastly, MC graphs are GRC graphs of the form $GR(\underbrace{m,m,\ldots,m}_\text{h})$.  
\end{remark}

Some of the basic properties of MC graphs that we will use in this paper are given in the next lemmas. 

\begin{lemma}[Stojmenovic,\cite{Stoj}]
The graph $MC(m^h)$ is vertex-regular with vertex-regularity $(2h-1)$ if $m=2$, and $2h$ if $m>2$.
\label{lemvreg}
\end{lemma}

\begin{remark}
Since $MC(m^h)$ is vertex-regular, it follows that $|E(MC(m^h))|=\frac{(2^h)(2h-1)}{2}$ if $m=2$ and $|E(MC(m^h))|=\frac{(m^h)(2h)}{2}$ if $m>2$.
\end{remark}

\begin{lemma}[Stojmenovic,\cite{Stoj}]
The graph $MC(m^h)$ is vertex-symmetric.
\label{lemvtran}
\end{lemma}

Vertex-symmetric graphs have a very nice property in terms of their distance matrices.  We define the distance matrix of a graph and state the very nice property of the distance matrix of a vertex-symmetric graph in the next subsection. 

\subsection{Graph's Distance Matrix, Distance Spectral Radius and Average Distance}

In this subsection, we briefly define the distance matrix, distance spectral radius and average distance of a graph and give some examples whenever possible. 

We first define the distance matrix of a graph. To do so, we need to define the concept of distance between vertices.

\begin{definition}
Let $\Gamma$ be a graph with $n$ vertices. The {\bf distance} between the vertices $v_i$ and $v_j$ in $V(\Gamma)$, denoted by $d(v_i,v_j)$, is the length of the shortest path between $v_i$ and $v_j$.
\end{definition}

We can now define the distance matrix of a graph.

\begin{definition}
The {\bf distance matrix} of a graph $\Gamma$ is the $n\times n$ matrix 
${\bf D}(\Gamma)=[d_{\Gamma}(v_i,v_j) ]_{v_i,v_j\in V(\Gamma)}.$
\end{definition}

Related to the concept of distance between vertices in a graph is the {\it diameter} of a graph. 

\begin{definition}
The {\bf diameter} of a graph $\Gamma$ denoted by $diam(\Gamma)$, is the maximum distance between any pair of vertices of $\Gamma$.
\end{definition}

\begin{example}
Let $\Gamma=Cay(\mathbb{Z}_8,\{1,2,4\})$. The distance matrix of $\Gamma$ is given by ${\bf D}(\Gamma)$
\begin{equation*}
{\bf D}(\Gamma)=
\begin{bmatrix}
0 & 1 & 1 & 2 & 1 & 2 & 1 & 1\\
1 & 0 & 1 & 1 & 2 & 1 & 2 & 1\\
1 & 1 & 0 & 1 & 1 & 2 & 1 & 2\\
2 & 1 & 1 & 0 & 1 & 1 & 2 & 1\\
1 & 2 & 1 & 1 & 0 & 1 & 1 & 2\\
2 & 1 & 2 & 1 & 1 & 0 & 1 & 1\\
1 & 2 & 1 & 2 & 1 & 1 & 0 & 1\\
1 & 1 & 2 & 1 & 2 & 1 & 1 & 0
\end{bmatrix}.
\end{equation*} 
Also, based on ${\bf D}(\Gamma)$, the diameter of $\Gamma$ is $2$. Note that in a circulant graph, the first row entries of the distance matrix represent the distance of the zero vertex to every other vertices in the graph. 
\label{samp}
\end{example}

\begin{remark}
Recall from Lemma \ref{lemvtran} that circulant graphs are vertex-transitive graphs. Since vertex-transitive graphs have circulant distance matrix \cite{Liu}, it follows that the distance matrix of a circulant graph is circulant.
\end{remark}

Next, we discuss the concept of graph {\it distance spectral radius}. The following definitions were taken from \cite{Liu}.

\begin{definition}
Let $v_i\in V(\Gamma)$, the{ \bf transmission of $v_i$ in $\Gamma$} denoted by $Tr_{\Gamma}(v_i)$, is the sum of distances from $v_i$ to all other vertices of $\Gamma$, that is
\begin{equation*}
Tr_{\Gamma}(v_i)=\sum_{v_j\in V(\Gamma)}d_{\Gamma}(v_i,v_j).
\end{equation*} 
\end{definition}

\begin{remark}
$Tr_{\Gamma}(v_i)$ is the row sum of $D(\Gamma)$ indexed by the vertex $v_i$.
\end{remark}

\begin{definition}
A graph $\Gamma$ is said to be {\bf s-transmission regular} if $Tr_{\Gamma}(v_i)=s$ for every $v_i\in V(\Gamma)$.  
\end{definition}

\begin{remark}
A circulant graph is an s-transmission regular graph.
\end{remark}

\begin{definition}
The largest eigenvalue of the distance matrix of a graph $\Gamma$ is called the {\bf distance spectral radius} of $\Gamma$ and is denoted by $\rho(\Gamma)$. 
\end{definition}

The distance spectral radius of a circulant graph is given in the next Lemma. 
 
\begin{lemma}[Lemma 2.5 Liu et al.,\cite{Liu}]
Let $\Gamma$ be a circulant graph and $\rho(\Gamma)$ be its distance spectral radius. Then
\begin{equation*}
\rho(\Gamma)=s.
\end{equation*}
\label{lem1} 
\end{lemma}   

\begin{example}
It follows from Lemma \ref{lem1} and Example \ref{samp} that the distance spectral radius of $Cay(\mathbb{Z}_8,\{1,2,4\})=9$. 
\end{example}

We end this subsection by discussing the concept of average distance in graph and by proving a very important theorem related to distance in circulant graph.

\begin{definition}
Let $\Gamma$ be a graph with $n$ vertices. The {\bf average distance} of $\Gamma$ denoted by $\mu(\Gamma)$ is the average of all distances in $\Gamma$. In symbol
\begin{equation*}
\mu(\Gamma)=\frac{1}{n(n-1)}\sum_{v_i,v_j\in V(\Gamma)}{d_{\Gamma}(v_i,v_j)}.
\end{equation*}
\end{definition}

\begin{example} 
The average distance of $Cay(\mathbb{Z}_8,\{1,2,4\})$ is $\frac{1}{56}({9\times 8})=\frac{9}{7}$.
\end{example}

\begin{theorem}
Let $\Gamma=Cay(\mathbb{Z}_n,S)$.  If $(0,v_1,v_2,\ldots,v_{k-1},j)$ is a shortest path from the $0$-vertex to vertex $j$ in $\Gamma$, then $(0,v_1^{-1},v_2^{-1},\ldots,v_{k-1}^{-1},j^{-1})$ is a shortest path from the $0$-vertex to vertex $j^{-1}$ in $\Gamma$.
\label{tsym}
\end{theorem}      

\begin{proof}
We first note that the path $(0,v_1,v_2,\ldots,v_{k-1},j)$ is a path from $0$-vertex to vertex $j$ in $\Gamma$ if and only if 
\begin{enumerate}[(i)]
\item each of the absolute value differences $|j-v_{k-1}|, |v_{k-1}-v_{k-2}|,\ldots, |v_2-v_1|,|v_1-0|$ belong to $S$ and
\item the sum $\displaystyle(j-v_{k-1})+\sum_{s=2}^{k-1}(v_s-v_{s-1})+(v_1-0)=j$.    
\end{enumerate}
Let $P=(0,v_1,v_2,\ldots,v_{k-1},j)$ be a shortest path from $0$-vertex to vertex $j$ in $\Gamma$. Since $P$ is a shortest path, then it is a path from $0$-vertex to vertex $j$ in $\Gamma$. Hence, we have
\begin{equation*}
|j-v_{k-1}|, |v_{k-1}-v_{k-2}|,\ldots, |v_2-v_1|,|v_1-0|\in S
\end{equation*} 
and  
\begin{equation*}
\bigg(\sum_{s=2}^{k-1}(v_s-v_{s-1})\bigg)+(j-v_{k-1})+(v_1-0)=j.
\end{equation*}

Now, consider the absolute value differences \begin{center}$|(n-j)-(n-v_{k-1})|,|(n-v_{k-1})-(n-v_{k-2})|,\ldots,|(n-v_2)-(n-v_1)|,|n-v_1|$. \end{center}

Since they are equivalent to the absolute value differences \begin{center} $|v_{k-1}-j|, |v_{k-2}-v_{k-1}|,\ldots, |v_1-v_2|,|n-v_1|$, \end{center}

and $|j-v_{k-1}|, |v_{k-1}-v_{k-2}|,\ldots, |v_2-v_1|,|v_1-0|\in S$, it follows that \begin{center} $|(n-j)-(n-v_{k-1})|,|(n-v_{k-1})-(n-v_{k-2})|,\ldots,|(n-v_2)-(n-v_1)|,|n-v_1|\in S$.\end{center} 

Moreover, the sum
\begin{equation*}
(v_{k-1}-j)+\sum_{s=2}^{k-1}(v_{s-1}-v_s)+(n-v_1)=n-j.
\end{equation*} 

Thus, $P^{-1}=(0,v_1^{-1},v_2^{-1},\ldots,v_{k-1}^{-1},j^{-1})$ is a path from $0$-vertex to vertex $j^{-1}$ in $\Gamma$. 
\vspace{2mm}
Next, we show that $P^{-1}$ is a shortest path from $0$-vertex to vertex $j^{-1}$ in $\Gamma$. We proceed by contradiction.  Suppose that there is another path $P^{-1*}=(0,v_1^{-1*},v_2^{-1*},\ldots,j^{-1})$ from $0$-vertex to vertex $j^{-1}$ in $\Gamma$ whose length is less than the length of $P^{-1}$. Then the path $P^{*}=(0,v_1^{*},v_2^{*},\ldots,j)$ is a path from $0$-vertex to vertex $j$ in $\Gamma$ with length less than the length of the path $P$, which is a shortest path from $0$-vertex to vertex $j$ in $\Gamma$; hence a contradiction. Thus, $P^{-1}$ is a shortest path from $0$-vertex to vertex $j^{-1}$ in $\Gamma$.
\end{proof}

\subsection{The Breadth-First Search Method}

The breadth-first search method simply called bfs method, is a method that finds the shortest paths from a given vertex of a graph to all the other vertices of the graph. The pseudo-code for the bfs algorithm is presented below and was taken from \cite{Gross}.     
\bigskip
\begin{center}
\fbox{\begin{minipage}{30em}
{\bf Breadth-first Search Algorithm}
\bigskip

{\it Input:} Undirected graph $\Gamma=(V(\Gamma),E(\Gamma))$ and a vertex $s\in V(\Gamma)$\\
{\it Output:} Breadth-first tree $T$ from $s$.\\
\hspace*{5mm} $V_i=\{\mbox{all vertices at distance $i$ from $s$}\}$\\
$V_0=\{s\}$\\
make $s$ the root of $T$\\
$i=0$\\
while $V_i\neq \emptyset$ do construct $V_{i+1}$\\
\hspace*{5mm} $V_{i+1}=\emptyset$\\
\hspace*{5mm} for each vertex $v\in V_i$ do\\
\hspace*{5mm} ``scan $v$"\\
\hspace*{10mm} for each edge $(v,w)$ do\\
\hspace*{15mm} if $w\notin \bigcup_j{V_j}$ then\\
\hspace*{15mm} make $w$ the next child of $v$ in $T$\\
\hspace*{15mm} add $w$ to $V_{i+1}$\\
\hspace*{5mm} $i=i+1$
\end{minipage}}
\end{center}

In this subsection of the paper, we simply illustrate how to determine the distance of the $0$-vertex in $MC(2^3)$ to every other vertices in $MC(2^3)$ using the bfs algorithm. For a detailed discussion on bfs algorithm, we suggest the references \cite{Gross,Gross2,Skiena}.

\begin{example}
We wish to determine the distance of the $0$-vertex in $MC(2^3)$ to every other vertices in $MC(2^3)$ using the bfs algorithm. We take $MC(2^3)$ as our graph and $0\in MC(2^3)$ as the start vertex. So we have for $i=0$, $V_0=\{0\}$ and the vertex $0$ will be the root of our tree $T$.

Since $V_0={0}$ is not empty, we construct $V_{0+1}=V_1=\emptyset$ in its initial state. Now, for $0\in V_0$, performing a scan on $0\in V_0$ yields, $V_1=\{1,2,4,6,7\}$. Since $V_1$ is not empty, we construct $V_{1+1}=V_2=\emptyset$ in its initial state. Performing a scan on $4\in V_1$ yields, $V_2=\{3,5\}$. Notice that at this stage, all the vertices in $\Gamma_{2^3}$ have been chosen. So performing a scan for the remaining vertices in $V_1$ leads to a repeated vertex, moreover performing the algorithm once more yields $V_3=\emptyset$. So we stop at this stage. The resulting tree is presented as the third figure in Figure \ref{imbfs8}.

\begin{figure}[h!]
\begin{center}
\includegraphics[width=0.6\columnwidth]{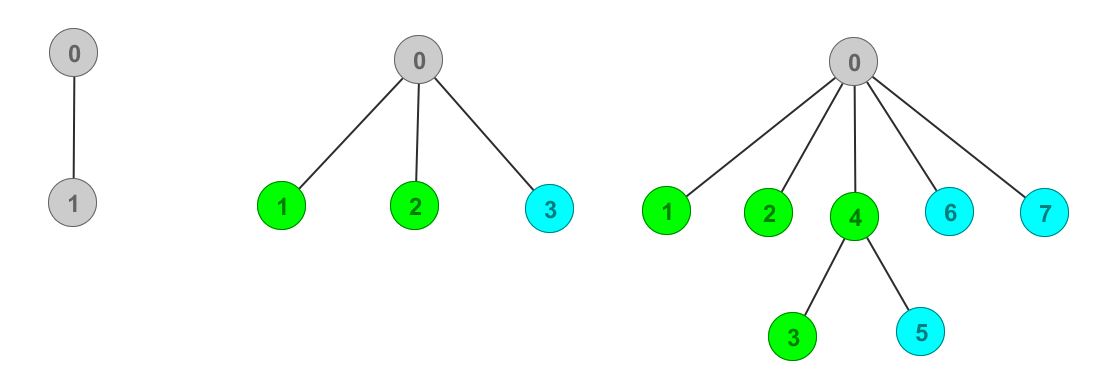}
\caption{{A bfs tree for graphs $MC(2^1)$, $MC(2^2)$ and $MC(2^3)$ with $0$-vertex as the start vertex. 
{\label{imbfs8}}%
}}
\end{center}
\end{figure}

From the resulting bfs tree, we see that $d_{MC(2^3)}(0,0)=0$, $d_{MC(2^3)}(0,v)=1$ if $v=1,2,4,6,7$ and $d_{MC(2^3)}(0,v)=2$ if $v=3,5$.     
\end{example}

\begin{remark}
In general, for the graph $MC(m^h)$, we have 
\begin{equation*}
V_1=\{m^0,m^1,\ldots,m^{h-2},m^{h-1},m^h-m^{h-1},\ldots,m^h-m^1,m^h-m^0\}.
\end{equation*}
Also, among all the vertices in $V_1$, we choose to scan $m^{h-1}$ first.
\end{remark}

In section 3, we will introduce a method for determining the distance of each vertices to the 0-vertex in $MC(2^h)$ as well as in $MC(3^h)$ that takes advantage of the result in Theorem \ref{tsym} as well as a particular property of the bfs tree of $MC(2^h)$ and $MC(3^h)$. In the mean time, we discuss the concept of vertex-forwarding index and edge-forwarding index of a graph in the next subsection.

\subsection{Graph's Forwarding Index}
  
For completeness, in this last subsection, we briefly discuss the concept of graph forwarding index. In what follows are some important definitions taken from \cite{Liu}.

\begin{definition}
A {\bf routing} $R$ of $\Gamma$ is a set of $n(n-1)$ elementary paths $R(x,y)$ specified for all ordered pairs $(x,y)$ of vertices of $\Gamma$.
\end{definition}

\begin{remark} The following are some of the usual notations and properties of routing in a graph.
\begin{enumerate}
\item If each of the paths specified by $R$ is shortest, the routing $R$ is said to be {\bf minimal}, denoted by $R_m$.
\item If $R(x,y)=R(y,x)$ specified by $R$, that is to say the path $R(y,x)$ is the reverse of the path $R(x,y)$ for all $x,y$, then the routing is {\bf symmetric}.
\item The set of all possible routing in a graph $\Gamma$ is denoted by $\cal{R}$$(\Gamma)$ and the subset of $\cal{R}$$(\Gamma)$ whose elements contains all the minimum routing in $\Gamma$ is denoted by $\cal{R}$$_m$$(\Gamma)$.   
\end{enumerate}
\end{remark}

\begin{definition}
Let $R\in\cal{R}$$(\Gamma)$ and $x\in V(\Gamma)$. The {\bf load of a vertex} $x$ in $R$ of $\Gamma$ denoted by $\xi_x(\Gamma,R)$ is the number of paths specified by $R$ passing through $x$ and admitting $x$ as an inner vertex.  
\end{definition}

\begin{definition}
The {\bf vertex-forwarding index of $\Gamma$ with respect to $R$}, denoted by $\xi(\Gamma,R)$  is the maximum number of paths of $R$ going through any vertex $x$ in $\Gamma$. Hence
\begin{equation*}
\xi(\Gamma,R)=\mbox{max}\{\xi_x(\Gamma,R):x\in V(\Gamma)\}.
\end{equation*}
\end{definition}

\begin{example}
Consider the graph $\Gamma$ shown in Figure \ref{imspec}. 

\begin{figure}[h!]
\begin{center}
\includegraphics[width=0.35\columnwidth]{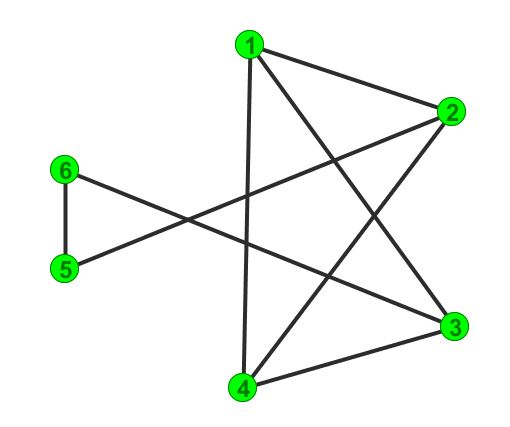}
\caption{{The graph $\Gamma$.
{\label{imspec}}%
}}
\end{center}
\end{figure}

The sets 
\begin{center}
$R_1=\{(1,2),(1,3),(1,4),(1,2,5),(1,3,6),(2,1),(2,1,3),(2,4),(2,5),(2,5,6),$
$(3,1),(3,1,2),(3,4),(3,6,5),(3,6),(4,1),(4,2),(4,3),(4,2,5),(4,3,6),$
$(5,2,1),(5,2),(5,6,3),(5,2,4),(5,6),(6,3,1),(6,5,2),(6,3),(6,3,4),(6,5)\}$.
\end{center} 
and
\begin{center}
$R_2=\{(1,4,2),(1,3),(1,4),(1,3,6,5),(1,3,6),(2,1),(2,1,3),(2,4),(2,5),(2,1,3,6),$
$(3,1),(3,1,2),(3,4),(3,4,1,2,5),(3,6),(4,1),(4,2),(4,3),(4,3,6,5),(4,3,6),(5,2,1),$
$(5,6,3,1,2),(5,6,3),(5,2,4),(5,6),(6,3,1),(6,5,2),(6,3),(6,3,4),(6,3,4,2,5)\}$.
\end{center}     
are routings of $\Gamma$. Observe that $R_1$ is a minimal routing while $R_2$ is not. Moreover, the load of vertex $3$ in $R_1$ of $\Gamma$ is $4$, that is $\xi_3(\Gamma,R_1)=4$. While the load of vertex $3$  in $R_2$ of $\Gamma$ is $9$, that is  $\xi_3(\Gamma,R_2)=9$.
\vspace{2mm} 

Finally, it can be verified that the load of each vertex in the routing $R_1$ of $\Gamma$ is given by: 1: 3, 2: 4, 3: 4, 4: 0, 5: 2, 6: 2. Hence, the forwarding index of $\Gamma$ with respect to $R_1$ is $4$. On the other hand, the load of each vertex in the routing $R_2$ of $\Gamma$ is given by: 1: 7, 2: 7, 3: 9, 4: 3, 5: 2, 6: 2. Hence, the forwarding index of $\Gamma$ with respect to $R_2$ is $9$.
\end{example}

\begin{definition}
The {\bf vertex-forwarding index of $\Gamma$}, denoted by $\xi(\Gamma)$ is the minimum forwarding index over all possible routing of $\Gamma$. In symbol,  
\begin{equation*}
\xi(\Gamma)=\mbox{min}\{\xi(\Gamma,R):R\in {\cal{R}}({\Gamma})\}.
\end{equation*}
\end{definition}

A similar definition corresponding to the edges of a graph is also discussed in \cite{Liu}. We define them also here for completeness.

\begin{definition}
The {\bf load of an edge} $e$ with respect to $R$, denoted by $\pi_e(\Gamma,R)$, is the number of the paths specified by $R$ going through it.
\end{definition}

\begin{definition}
The {\bf edge forwarding index of a graph $\Gamma$ with respect to a routing $R$}, denoted by $\pi(\Gamma,R)$ is the maximum number of paths specified by $R$ going through any edge of $\Gamma$. Hence
\begin{equation*}
\pi(\Gamma,R)=\mbox{max}\{\pi_e(\Gamma,R):e\in E(\Gamma)\}.
\end{equation*}
\end{definition}

\begin{definition}
The {\bf edge-forwarding index of a graph $\Gamma$}, denoted by $\pi(\Gamma)$ is defined by
\begin{equation*}
\pi(\Gamma)=\mbox{min}\{\pi(\Gamma,R):R\in {\cal{R}}(\Gamma)\}.
\end{equation*}
\end{definition}

\begin{remark}
Since finding all possible routing in a graph $\Gamma$ is a tedious task to do, in this paper, we are guided by some known results on graph's vertex and edge-forwarding index in order to obtain new essential results. 
\end{remark}

\begin{remark}
For routings of shortest path, Liu et al. \cite{Liu}, define
\begin{equation*}
\xi_m(\Gamma)=\mbox{min}\{\xi(\Gamma,R_m):R_m\in {\cal{R}}_m({\Gamma})\}
\end{equation*}
and
\begin{equation*}
\pi_m(\Gamma)=\mbox{min}\{\pi(\Gamma,R_m):R_m\in {\cal{R}}_m(\Gamma)\}.
\end{equation*}
\end{remark}

We are now ready to enumerate some important results in \cite{Liu} that we will be using in this paper.

\begin{lemma}[Lemma 4.2 Liu et al., \cite{Liu}]
If $\Gamma$ is a connected circulant graph of order $n$, then
\begin{equation*}
 \xi(\Gamma)=\xi_m(\Gamma)=\rho(\Gamma)-(n-1).
\end{equation*}
\label{lxi}
\end{lemma} 

\begin{lemma}[Lemma 4.5 Liu et al., \cite{Liu}]
If $\Gamma$ is a connected $r-$regular circulant graph of order $n$, then
\begin{equation*}
\frac{2\rho(\Gamma)}{r}\leq \pi(\Gamma)\leq n+\rho(\Gamma)-(2r-1).
\end{equation*}
\label{lpi}
\end{lemma}

In the next section, we finally present our results.

\section{Main Results}  
In this section, we determine the distance matrix of MC graphs of order power of two and three. As a consequence, the exact values of the diameters, distance spectral radii, average distances, and vertex-forwarding indices of the graphs were determined. Surprisingly, the distance spectral radii of these two MC graphs are exactly the sequences A045883 and  A212697 in The On-line Encyclopedia of Integer Sequences (OEIS). Lastly, some bounds on the edge-forwarding indices of the graphs are presented.

\subsection{BFS Tree and Diameter}
From here forward, we denote the graph $MC(m^h)$ by $\Gamma_{m^h}$ and by $bfs(\Gamma_{m^h})$ the bfs tree of the graph $MC(m^h)$. A method for constructing $bfs(\Gamma_{m^h})$ based on $bfs(\Gamma_{m^{h-1}})$ for $m=2$ and $3$ is presented in this subsection. 

For $m=2$, the method is based on the following properties of $\Gamma_{2^h}$: 
\begin{enumerate}
    \item As a consequence of Theorem \ref{tsym}, in order to determine the distance of each vertices to the 0-vertex in $\Gamma_{2^h}$, it is enough to consider the vertices $0,1,\ldots,2^{h-1}$.
    \item Let $A=\{2^{h-1}-2^{h-3},2^{h-1}-2^{h-4},\ldots,2^{h-1}-2^{h-h}\}$. For each $a\in A$, we have
    \begin{equation*}
        d_{\Gamma_{2^h}}(0,a)=d_{\Gamma_{2^{h-1}}}(0,a)+1.
    \end{equation*}
    \item Since $V(\Gamma_{2^{h-1}})\subset V(\Gamma_{2^h})$ and $S_{\Gamma_{2^h}}=S_{\Gamma_{2^{h-1}}}\bigcup \{2^{h-1}\}$, the parent-child relationship for $bfs(\Gamma_{2^{h-1}})$ is the same as in the parent-child relationship for $bfs(\Gamma_{2^{h}})$ for parents $2^0,2^1,\ldots,2^{h-2}$. 
\end{enumerate}

Before presenting the method, we introduce some terms that will simplify our discussion.

\begin{definition}
The {\bf left part} $L$ of $bfs(\Gamma_{2^h})$ refers to the vertices $2^0,2^1,\ldots, 2^{h-2}$, their children and grandchildren. While the {\bf middle part} $M$ of $bfs(\Gamma_{2^h})$ refers to the vertex $2^{h-1}$, its children and grandchildren. Finally, the {\bf right part} $R$ of $bfs(\Gamma_{2^h})$ refers to the vertices $2^h-2^{h-2},2^h-2^{h-3},\ldots, 2^h-2^{h-h}$, their children and grandchildren.   
\end{definition}

\begin{example}
From Figure \ref{imbfs8}, we see that the left part of $bfs(\Gamma_{2^3})$ are the vertices 1 and 2. While the middle part of $bfs(\Gamma_{2^3})$ are the vertices 3, 4 and 5. Finally, the right part of $bfs(\Gamma_{2^3})$ are the vertices 6 and 7. 
\end{example}
We are now ready to present the method. 

\begin{center}
\fbox{\begin{minipage}{30em}
{\bf Method on Constructing a BFS Tree for $MC(2^h)$}
\bigskip

Let $bfs(\Gamma_{2^{h-1}})$ be a bfs tree for $\Gamma_{2^{h-1}}$, a bfs tree for $\Gamma_{2^h}$ based on $bfs(\Gamma_{2^{h-1}})$ can be constructed as follows:
\begin{enumerate}
    \item In $bfs(\Gamma_{2^{h-1}})$, replace the 0-vertex by $2^{h-1}$.
    \item Descend the vertex $2^{h-1}$ and right part of $bfs(\Gamma_{2^{h-1}})$ by a unit and introduce the new 0-vertex.
    \item Complete $bfs(\Gamma_{2^h})$ using Theorem \ref{tsym}. 
\end{enumerate}
\end{minipage}}
\end{center}

\begin{example}
Using the proposed method, we can construct $bfs(\Gamma_{2^4})$ and $bfs(\Gamma_{2^5})$ starting from $bfs(\Gamma_{2^3})$ in Figure \ref{imbfs8}. The resulting trees are shown in Figure \ref{imbfs16}.

\begin{figure}[h!]
\begin{center}
\includegraphics[width=0.5\columnwidth]{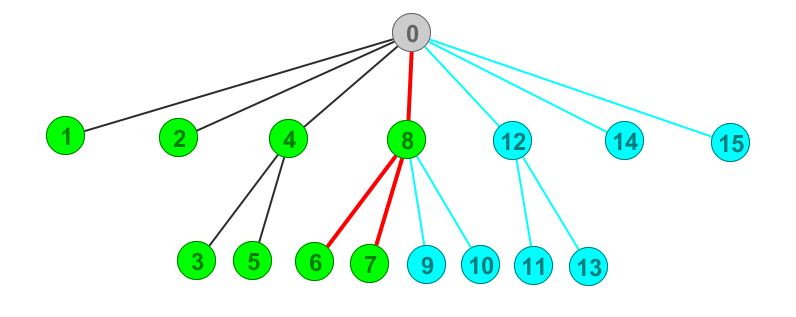}\\
\includegraphics[width=0.6\columnwidth]{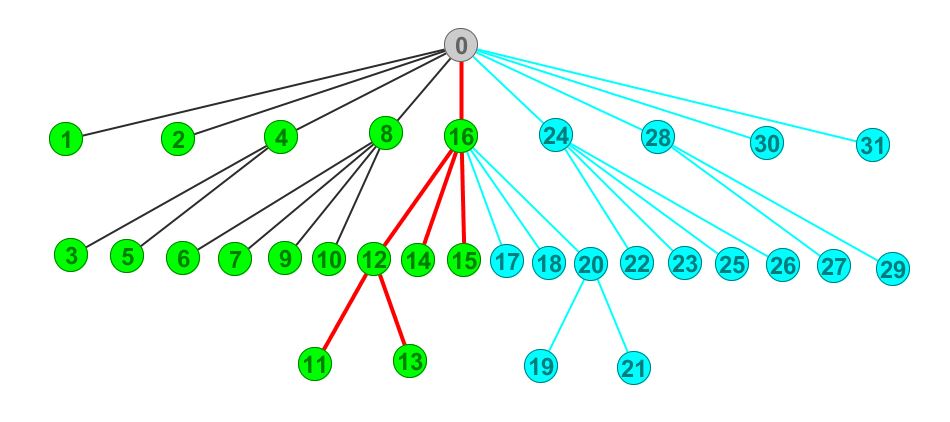}
\caption{{A bfs tree for graphs $\Gamma_{2^4}$ and $\Gamma_{2^5}$. The green-colored vertices in $bfs(\Gamma_{2^4})$ refer to the vertices that originally appeared in $bfs(\Gamma_{2^3})$ while the green-colored vertices in $bfs(\Gamma_{2^5})$ refer to the vertices that originally appeared in $bfs(\Gamma_{2^4})$. The green-colored vertices with edge appeared in red refer to the descended vertices in $bfs(\Gamma_{2^3})$ and $bfs(\Gamma_{2^4})$ respectively. The blue-colored vertices with blue-colored edges are the add-on vertices that completes $bfs(\Gamma_{2^4})$ and $bfs(\Gamma_{2^5})$ using Theorem \ref{tsym}.  
{\label{imbfs16}}%
}}
\end{center}
\end{figure}
\end{example}

\begin{remark}
Given $bfs(\Gamma_{2^5})$, we can construct $bfs(\Gamma_{2^6})$,$bfs(\Gamma_{2^7})\ldots$, $bfs(\Gamma_{2^h})$ and so on. 
\end{remark}

Based on the construction of $bfs(\Gamma_{2^h})$, we have

\begin{theorem}
Let $h$ be a positive integer. Then
\[
d_{\Gamma_{2^h}}(0,j)=
\begin{cases}
d_{\Gamma_{2^{h-1}}}(0,j)\ \ \ \ \ \ \ \ \ \mbox{if $j=0,1,2,\ldots,\mbox{max}[L(bfs(\Gamma_{2^{h-1}}))\bigcup M(bfs(\Gamma_{2^{h-1}}))]$}\\
d_{\Gamma_{2^{h-1}}}(0,j)+1\ \ \mbox{if $j= \mbox{min}[R(bfs(\Gamma_{2^{h-1}}))],\ldots,2^{h-1}$.}
\end{cases}
\]  
\label{dis2}
\end{theorem}

Table \ref{tabmax} provides the value of $\mbox{max}[L(bfs(\Gamma_{2^h})\bigcup M(bfs(\Gamma_{2^h})))]$ and $\mbox{min}[R(bfs(\Gamma_{2^h}))]$ for $h=1,2,\ldots, 10$.  

\begin{table}[h!]
\centering
\scalebox{0.9}{
\begin{tabular}{ |c|c|c|}
\hline
$h$ in $\Gamma_{2^h}$ & $\mbox{max}[L(bfs(\Gamma_{2^h})\bigcup M(bfs(\Gamma_{2^h})))]$ & $\mbox{min}[R(bfs(\Gamma_{2^h}))]$\\
\hline
1 & 1 & \\
\hline
2 & 2 & 3\\
\hline
3 & 5 & 6\\
\hline
4 & 10 & 11\\
\hline
5 & 21 & 22\\
\hline
6 & 42 & 43\\
\hline
7 & 85 & 86\\
\hline
8 & 170 & 171\\
\hline
9 & 341 & 342\\
\hline
10 & 682 & 683\\
\hline
\end{tabular}}
\caption{{ Value of $\mbox{max}[L(bfs(\Gamma_{2^h})\bigcup M(bfs(\Gamma_{2^h})))]$ and $\mbox{min}[R(bfs(\Gamma_{2^h}))]$ for $h=1,2,\ldots, 10$.
{\label{tabmax}}
}}
\end{table}

Notice the following properties of $\mbox{max}[L(bfs(\Gamma_{2^h})\cup M(bfs(\Gamma_{2^h})))]$:

\begin{equation*}
\mbox{max}[L(bfs(\Gamma_{2^h})\cup M(bfs(\Gamma_{2^h})))]=\mbox{min}[R(bfs(\Gamma_{2^h}))]-1
\end{equation*}

\begin{equation*}
\mbox{max}[L(bfs(\Gamma_{2^h})\cup M(bfs(\Gamma_{2^h})))]=2^h-\mbox{min}[R(bfs(\Gamma_{2^{h-1}}))]   
\end{equation*}

\[
\mbox{max}[L(bfs(\Gamma_{2^{h-1}})\cup M(bfs(\Gamma_{2^{h-1}})))]=
\begin{cases}
\frac{1}{3}(4^n-1)\ \ \ \ \ \ \  \mbox{if $h=2n$}\\
\frac{2}{3}(4^n-1)\ \ \ \ \ \ \ \mbox{if $h=2n+1$.}
\end{cases}
\]  

Using the three properties of $\mbox{max}[L(bfs(\Gamma_{2^h})\cup M(bfs(\Gamma_{2^h})))]$, Theorem \ref{dis2} can be rewritten more explicitly.

\begin{theorem}
Let $h$ be a positive integer.\\ 
If $h=2n$ then
\[
d_{\Gamma_{2^h}}(0,j)=
\begin{cases}
d_{\Gamma_{2^{h-1}}}(0,j)\ \ \ \ \ \ \ \ \ \mbox{if $j=0,1,2,\ldots,\frac{1}{3}(4^n-1)$}\\
d_{\Gamma_{2^{h-1}}}(0,j)+1\ \ \mbox{if $j= \frac{1}{3}(4^n-1)+1,\ldots,2^{h-1}$.}
\end{cases}
\]
\\
If $h=2n+1$ then
\[
d_{\Gamma_{2^h}}(0,j)=
\begin{cases}
d_{\Gamma_{2^{h-1}}}(0,j)\ \ \ \ \ \ \ \ \ \mbox{if $j=0,1,2,\ldots,\frac{2}{3}(4^n-1)$}\\
d_{\Gamma_{2^{h-1}}}(0,j)+1\ \ \mbox{if $j= \frac{2}{3}(4^n-1)+1,\ldots,2^{h-1}$.}
\end{cases}
\]  
\label{thdis}
\end{theorem}

\begin{example}
Given the first row of the distance matrix of the graph $\Gamma_{2^3}$ which is $0\ 1\ 1\ 2\ 1\ 2\ 1\ 1$, we can determine the first row of the distance matrix of the graphs $\Gamma_{2^h}$ for $h=4,5,\ldots$ using Theorem \ref{thdis}.

For $\Gamma_{2^4}$, since $4=2(2)$, Theorem \ref{thdis} says that for vertices $j=1,\ldots, 5$ we have $d_{\Gamma_{2^4}}(0,j)=d_{\Gamma_{2^3}}(0,j)$ while for $j=6,7$ we have $d_{\Gamma_{2^4}}(0,j)=d_{\Gamma_{2^3}}(0,j)+1$. Hence we have $0\ 1\ 1\ 2\ 1\ 2\ 2\ 2$ as the first 8 entries of the first row of the distance matrix of $\Gamma_{2^4}$. We can then complete the remaining distances using Theorem \ref{tsym}. The first row of the distance matrix of $\Gamma_{2^4}$ is $0\ 1\ 1\ 2\ 1\ 2\ 2\ 2\ 1\ 2\ 2\ 2\ 1\ 2\ 1\ 1$. This can be verified by considering the bfs tree for $\Gamma_{2^4}$. 
\end{example}

We now state a result involving the diameter of $\Gamma_{2^h}$. The result follows immediately from the ``descend" action stated in the second step of the propose method for the construction of the bfs tree for $\Gamma_{2^h}$.

\begin{theorem}
The diameter of $\Gamma_{2^h}$ denoted by $diam(\Gamma_{2^h})$ has the property
\[
diam(\Gamma_{2^h})=
\begin{cases}
diam(\Gamma_{2^{h-1}})\ \ \ \ \ \ \ \ \ \ \ \ \ \mbox{if $h$ is even}\\
diam(\Gamma_{2^{h-1}})+1\ \ \ \ \ \ \ \mbox{if $h$ is odd.}
\end{cases}
\]  
\label{2diam}
\end{theorem}

The next result follows immediately from Theorem \ref{2diam} and the fact that $diam(\Gamma_{2^1})=1$.

\begin{corollary}
The diameter of $\Gamma_{2^h}$ denoted by $diam(\Gamma_{2^h})$ is given by
\[
diam(\Gamma_{2^h})=
\begin{cases}
n\ \ \ \ \ \ \ \ \ \ \ \ \ \ \mbox{if $h=2n$}\\
n+1\ \ \ \ \ \ \ \mbox{if $h=2n+1$.}
\end{cases}
\]  
\label{cordiam2}
\end{corollary}

\begin{remark}
Another proof for Corollary \ref{cordiam2} was given by Arno and Wheeler \cite{Arno} and was stated in \cite{Stoj}. A general formula for the diameter of any generalized recursive circulant graph is given in (Theorem 4 Tang et al., \cite{Tang}) whose proof depends on the proof of a result in (Theorem 6 Stojmenovic, \cite{Stoj}). The diameter of $GR(m_h,m_{h-1},\ldots,m_1)$ is given by 
\begin{equation}
    \left(\sum_{i=1}^h \left\lfloor\frac{m_i}{2}\right\rfloor \right)-\left\lfloor\frac{\lambda}{2}\right\rfloor
\label{eqtang}
\end{equation}
where $\lambda$ is the number of even dimensions in $(m_h,m_{h-1},\ldots,m_1)$. Using the given formula to $GR({\bf 2}_h)$ we have
\begin{align*}
    \displaystyle diam(\Gamma_{2^h})&=\left(\sum_{i=1}^h \left\lfloor\frac{2}{2}\right\rfloor \right)-\left\lfloor\frac{h}{2}\right\rfloor\\
    &=h-\left\lfloor\frac{h}{2}\right\rfloor\\
    &=\begin{cases}
n\ \ \ \ \ \ \ \ \ \ \ \ \ \ \mbox{if $h=2n$}\\
n+1\ \ \ \ \ \ \ \mbox{if $h=2n+1$.}
\end{cases}
\end{align*}
\end{remark}

We now consider the bfs tree for $\Gamma_{3^h}$. The method for constructing $bfs(\Gamma_{m^h})$ for $m=3$ is based on the following properties of $\Gamma_{3^h}$

\begin{enumerate}
    \item As a consequence of Theorem \ref{tsym}, in order to determine the distance of each vertices to the 0-vertex in $\Gamma_{3^h}$, it is enough to consider the vertices $0,1,\ldots,\frac{3^{h}-1}{2}$.
    \item Let $A=\{3^{h-1}-3^{h-2},3^{h-1}-3^{h-3},\ldots,3^{h-1}-3^{h-h}\}$. For each $a\in A$, we have
    \begin{equation*}
        d_{\Gamma_{3^h}}(0,a)=d_{\Gamma_{3^{h-1}}}(0,a)+1.
    \end{equation*}
    \item Since $V(\Gamma_{3^{h-1}})\subset V(\Gamma_{3^h})$ and $S_{\Gamma_{3^h}}=S_{\Gamma_{3^{h-1}}}\bigcup \{3^{h-1}\}$, the parent-child relationship for $bfs(\Gamma_{3^{h-1}})$ is the same as in the parent-child relationship for $bfs(\Gamma_{3^{h}})$ for parents $3^0,3^1,\ldots,3^{h-2}$.
    \item For $b=1,2,\ldots,\frac{3^{h-1}-1}{2}$, we have
    \begin{equation*}
        d_{\Gamma_{3^{h-1}}}(0,b)=d_{\Gamma_{3^{h}}}(3^{h-1},3^{h-1}+b).
    \end{equation*}
\end{enumerate}

Before presenting the method, we define some terms involving ``parts" of the bfs tree of the graph $\Gamma_{m^h}$ where $m>2$ is odd.

\begin{definition}
Let $m>2$. The {\bf left part} $L$ of $bfs(\Gamma_{m^h})$ refers to the vertices $m^0,m^1,\ldots, m^{h-1}$, their children and grandchildren. While the {\bf right part} $R$ of $bfs(\Gamma_{m^h})$ refers to the vertices $m^h-m^{h-1},m^h-m^{h-2},\ldots, m^h-m^{h-h}$, their children and grandchildren.     
\end{definition}

We now present the method.

\begin{center}
\fbox{\begin{minipage}{30em}
{\bf Method on Constructing a BFS Tree for $MC(3^h)$}
\bigskip

Let $bfs(\Gamma_{3^{h-1}})$ be a bfs tree for $\Gamma_{3^{h-1}}$, a bfs tree for $\Gamma_{3^h}$ based on $bfs(\Gamma_{3^{h-1}})$ can be constructed as follows:
\begin{enumerate}
    \item In $bfs(\Gamma_{3^{h-1}})$, replace the 0-vertex by $3^{h-1}$.
    \item Descend the vertex $3^{h-1}$ and right part of $bfs(\Gamma_{3^{h-1}})$ by a unit and introduce the new 0-vertex.
    \item Reproduce the left part of $bfs(\Gamma_{3^{h-1}})$ with the substitution \begin{center} $0:=3^{h-1}$.\end{center}
    \item Complete $bfs(\Gamma_{3^h})$ using Theorem \ref{tsym}. 
\end{enumerate}
\end{minipage}}
\end{center}

\begin{example}
We illustrate the method by constructing $bfs(\Gamma_{3^3})$. Given in Figure \ref{a} are the bfs tree for $\Gamma_{3^1}$ and $\Gamma_{3^2}$ respectively. 

\begin{figure}[h!]
\begin{center}
\includegraphics[width=0.5\columnwidth]{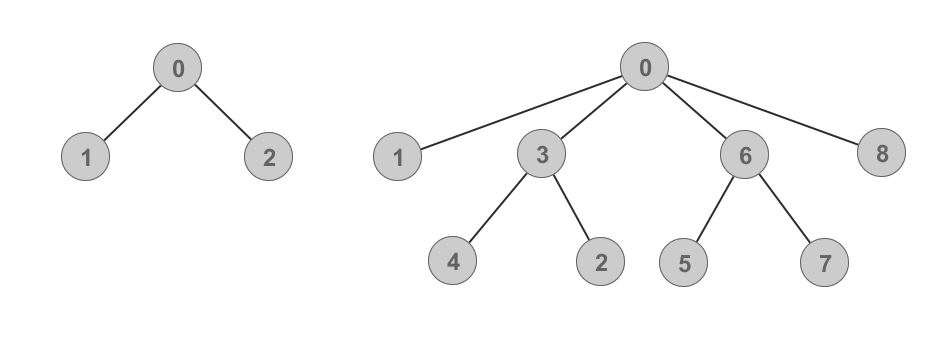}
\caption{{bfs tree for graphs $\Gamma_{3^1}$, and $\Gamma_{3^2}$ with $0$ as the start vertex. 
{\label{a}}%
}}
\end{center}
\end{figure}

Using the method presented above and by taking $bfs(\Gamma_{3^2})$ as an input, we get $bfs(\Gamma_{3^3})$ shown in Figure \ref{b}. 
\end{example}

Based on the construction of $bfs(\Gamma_{3^h})$, we have

\begin{theorem}
Let $h$ be a positive integer. Then
\[
d_{\Gamma_{3^h}}(0,j)=
\begin{cases}
d_{\Gamma_{3^{h-1}}}(0,j)\ \ \ \ \ \ \ \ \ \mbox{if $j=0,1,2,\ldots,\mbox{max}[L(bfs(\Gamma_{3^{h-1}}))]$}\\
d_{\Gamma_{3^{h-1}}}(0,j)+1\ \ \mbox{if $j= \mbox{min}[R(bfs(\Gamma_{2^{h-1}}))],\ldots,3^{h-1}-1$.}
\end{cases}
\]  
Moreover, if $k\in V(\Gamma_{3^h})$ such that $k=3^{h-1}+j$ for $j=0,1,\ldots,\frac{3^{h-1}-1}{2}$ then
\begin{equation*}
    d_{\Gamma_{3^h}}(0,k)=d_{\Gamma_{3^{h-1}}}(0,j)+1.
\end{equation*}
\label{tbfs3}
\end{theorem}

\begin{figure}[h!]
\begin{center}
\includegraphics[width=0.6\columnwidth]{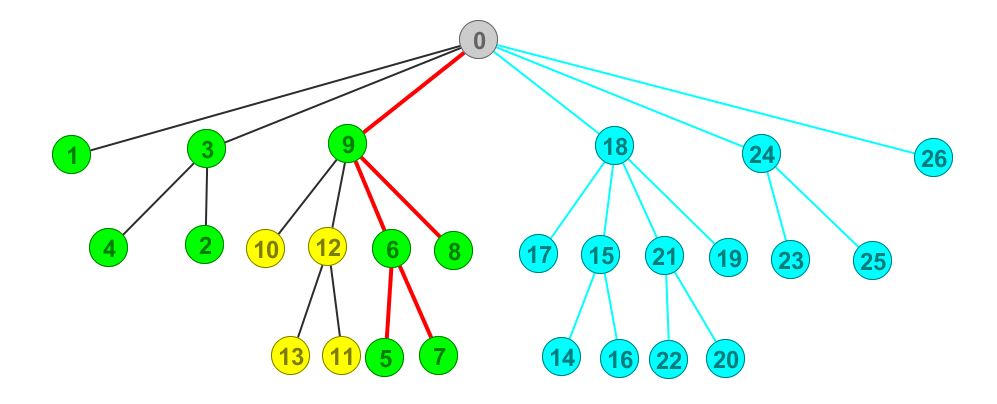}
\caption{{A bfs tree for graph $\Gamma_{3^3}$. The green-colored vertices in $bfs(\Gamma_{3^3})$ refer to the vertices that originally appeared in $bfs(\Gamma_{3^2})$. The green-colored vertices with edge appeared in red refer to the descended vertices in $bfs(\Gamma_{3^2})$. The yellow-colored vertices in $bfs(\Gamma_{3^3})$ refer to the resulting vertices as a result of  reproducing the left part of $bfs(\Gamma_{3^{h-1}})$ with the substitution $0:=3^{h-1}$. Finally, the blue-colored vertices with blue-colored edges are the add-on vertices that completes $bfs(\Gamma_{3^3})$ using Theorem \ref{tsym}.  
{\label{b}}%
}}
\end{center}
\end{figure}

Theorem \ref{tbfs3} can be restated explicitly by noting that 
\begin{equation*}
    \mbox{max}[L(bfs(\Gamma_{3^{h-1}}))]=\frac{3^{h-1}-1}{2}
\end{equation*}
and 
\begin{equation*}
    \mbox{min}[R(bfs(\Gamma_{2^{h-1}}))]=\frac{3^{h-1}-1}{2}+1.
\end{equation*}

Hence we have

\begin{theorem}
Let $h$ be a positive integer. Then
\[
d_{\Gamma_{3^h}}(0,j)=
\begin{cases}
d_{\Gamma_{3^{h-1}}}(0,j)\ \ \ \ \ \ \ \ \ \mbox{if $j=0,1,2,\ldots,\frac{3^{h-1}-1}{2}$}\\
d_{\Gamma_{3^{h-1}}}(0,j)+1\ \ \mbox{if $j=\frac{3^{h-1}-1}{2}+1,\ldots,3^{h-1}-1$.}
\end{cases}
\]  
Moreover, if $k\in V(\Gamma_{3^h})$ such that $k=3^{h-1}+j$ for $j=0,1,\ldots,\frac{3^{h-1}-1}{2}$ then
\begin{equation*}
    d_{\Gamma_{3^h}}(0,k)=d_{\Gamma_{3^{h-1}}}(0,j)+1.
\end{equation*}
\label{tbfs3b}
\end{theorem}

\begin{example}
We can use Theorem \ref{tbfs3b} to determine the first row of the distance matrix of the graph $\Gamma_{3^h}$ given the first row of the distance matrix of the graph $\Gamma_{3^{h-1}}$. For instance, given the first row of the distance matrix of the graph $\Gamma_{3^3}$ which can easily obtained from $bfs(\Gamma_{3^3})$ in Figure \ref{b}, we have 

\begin{center}
    {\color{green}{01212323212323}}\ {\color{red}{4343234343232}}\ {\color{yellow}{12323434323434}}\\ \color{blue}43432343432321 \ 2323434323434 \ 3232123232121 
\end{center}

as the first row of the distance matrix of the graph $\Gamma_{3^4}$. The numbers colored-green refer to the distance of the vertices $0$ up to $13$ to the vertex $0$. While the red-colored numbers represent the distance of vertices $14$ up to $26$ to the vertex $0$. The yellow-colored numbers represent the distance of vertices $27$ up to $40$ to the vertex $0$. Finally, the blue-colored numbers represent the distance of the vertices $41$ up to $80$ to the vertex $0$ and were obtained using Theorem \ref{tsym}.  
\end{example}

For the diameter of graph $\Gamma_{3^h}$, the next result follows from the ``descend" action stated in the second step of the proposed method for constructing the bfs tree of $\Gamma_{3^h}$.

\begin{theorem}
The diameter of $\Gamma_{3^h}$ denoted by $diam(\Gamma_{3^h})$ has the property 
\begin{equation*}
    diam(\Gamma_{3^h})=diam(\Gamma_{3^{h-1}})+1.
\end{equation*}
\label{thdiam3}
\end{theorem}

The last result in this subsection follows immediately from Theorem \ref{thdiam3} and the fact that $diam(\Gamma_{3^1})=1$.

\begin{corollary}
The diameter of $\Gamma_{3^h}$ denoted by $diam(\Gamma_{3^h})$ is $h$.
\label{cordiam3}
\end{corollary}

Corollary \ref{cordiam3} agrees with the result of Tang et al. \cite{Tang} in equation \eqref{eqtang} as well as with the result of Wong and Coppersmith \cite{Wong}.  

\subsection{Distance Spectral Radius and Two OEIS Sequence}

Denote by $\rho(\Gamma_{m^{h-1}})$ and $\rho(\Gamma_{m^{h}})$ the distance spectral radius of the graphs $\Gamma_{m^{h-1}}$ and $\Gamma_{m^{h}}$ respectively. For $m=2$, the next result relating $\rho(\Gamma_{2^{h-1}})$ and $\rho(\Gamma_{2^{h}})$  follows immediately from the proposed construction for the bfs tree of $\Gamma_{2^h}$.

\begin{lemma}
The two distance spectral radii $\rho(\Gamma_{2^{h-1}})$ and $\rho(\Gamma_{2^{h}})$ are related by the equation
\begin{equation*}
    \rho(\Gamma_{2^{h}})=2\left(\rho(\Gamma_{2^{h-1}})+|R[bfs(\Gamma_{2^{h-1}})]|+1\right)-1.
\end{equation*}
\end{lemma}

Using the relationship

\[
\mbox{max}[L(bfs(\Gamma_{2^{h-1}})\cup M(bfs(\Gamma_{2^{h-1}})))]=
\begin{cases}
\frac{1}{3}(4^n-1)\ \ \ \ \ \ \  \mbox{if $h=2n$}\\
\frac{2}{3}(4^n-1)\ \ \ \ \ \ \ \mbox{if $h=2n+1$,}
\end{cases}
\]  

A more explicit relationship for $\rho(\Gamma_{2^{h-1}})$ and $\rho(\Gamma_{2^{h}})$ can be stated.

\begin{lemma}
The two distance spectral radii $\rho(\Gamma_{2^{h-1}})$ and $\rho(\Gamma_{2^{h}})$ are related by the equation
\begin{equation*}
    \rho(\Gamma_{2^{h}})=2\left(\rho(\Gamma_{2^{h-1}})+(2^{h-1}-\frac{1}{3}(4^n-1))\right)-1,
\end{equation*}
if $h=2n$.
While
\begin{equation*}
    \rho(\Gamma_{2^{h}})=2\left(\rho(\Gamma_{2^{h-1}})+(2^{h-1}-\frac{2}{3}(4^n-1))\right)-1,
\end{equation*}
if $h=2n+1$.
\label{rel}
\end{lemma}

Finally, an exact formula for the distance spectral radius of $\Gamma_{2^h}$ that depends on the value of $h$ is given in the next result.

\begin{theorem}
For all positive integer $h$, we have
\begin{equation*}
    \rho(\Gamma_{2^h})=\frac{2^h(3h+1)-(-1)^h}{9}.
\end{equation*}
\label{thspec2}
\end{theorem}

\begin{proof}
For the basis step, observe that when $h=1$, we have 
\begin{equation*}
    \rho(\Gamma_{2^1})=1=\frac{2^1(3(1)+1)-(-1)^1}{9}.
\end{equation*}
and when $h=2$, we have
\begin{equation*}
    \rho(\Gamma_{2^2})=3=\frac{2^2(3(2)+1)-(-1)^2}{9}.
\end{equation*}

We prove the theorem by considering two cases. The first case is when $h$ is even. Let $h$ be an even integer and suppose that for all $k<h$ we have
\begin{equation*}
    \rho(\Gamma_{2^k})=\frac{2^k(3k+1)-(-1)^k}{9}.
\end{equation*}
We show that 
\begin{align*}
    \rho(\Gamma_{2^h})&=\frac{2^h(3h+1)-(-1)^h}{9}.\\
    &=\frac{2^h}{9}(3h+1)-\frac{1}{9}.
\end{align*}
Since $h$ is even, then $h=2n$ for some positive integer $n$. Using Lemma \ref{rel}, we have
\begin{equation*}
    \rho(\Gamma_{2^{h}})=2\left(\rho(\Gamma_{2^{h-1}})+(2^{h-1}-\frac{1}{3}(4^{\frac{h}{2}}-1))\right)-1.
\end{equation*}
Moreover, since $h-1<h$, by our induction hypothesis we have
\begin{align*}
    \rho(\Gamma_{2^h})&=2\left(\frac{2^{h-1}(3(h-1)+1)+1}{9}+2^{h-1}-\frac{1}{3}(4^{\frac{h}{2}-1})\right)-1\\
    &=\frac{2h}{9}(3h-2)+\frac{2}{9}+2^h-\frac{2}{3}(2^h-1)-1\\
    &=\frac{2^h}{9}(3h-2)+\frac{2h}{9}(9)-\frac{2h}{9}(6)+\frac{2}{9}+\frac{6}{9}-\frac{9}{9}\\
    &=\frac{2^h}{9}(3h-2+9-6)-\frac{1}{9}\\
    &=\frac{2^h}{9}(3h+1)-\frac{1}{9}.
\end{align*}
For the other case, let $h$ be an odd integer and suppose that for all $k<h$ we have
\begin{equation*}
    \rho(\Gamma_{2^h})=\frac{2^k(3k+1)-(-1)^k}{9}.
\end{equation*}
We show that 
\begin{align*}
    \rho(\Gamma_{2^h})&=\frac{2^h(3h+1)-(-1)^h}{9}.\\
    &=\frac{2^h}{9}(3h+1)+\frac{1}{9}.
\end{align*}
Since $h$ is odd, then $h=2n+1$ for some positive integer $n$. Using Lemma \ref{rel}, we have
\begin{equation*}
\rho(\Gamma_{2^{h}})=2\bigg(\rho(\Gamma_{2^{h-1}})+(2^{h-1}-\frac{2}{3}(4^{\frac{h-1}{2}}-1))\bigg)-1.
\end{equation*}
Moreover, since $h-1<h$, by our induction hypothesis we have
\begin{align*}
    \rho(\Gamma_{2^h})&=2\left(\frac{2^{h-1}(3(h-1)+1)-1}{9}+2^{h-1}-\frac{2}{3}(4^{\frac{h-1}{2}-1})\right)-1\\
    &=\frac{2h}{9}(3h-2)-\frac{2}{9}+2^h-\frac{4}{3}(2^{h}-1-1)-1\\
    &=\frac{2^h}{9}(3h-2)+\frac{2h}{9}(9)-\frac{2h}{9}(6)-\frac{2}{9}+\frac{12}{9}-\frac{9}{9}\\
    &=\frac{2^h}{9}(3h-2+9-6)+\frac{1}{9}\\
    &=\frac{2^h}{9}(3h+1)+\frac{1}{9}.
\end{align*} 
Hence for all positive integer $h$ we have
\begin{equation*}
    \rho(\Gamma_{2^h})=\frac{2^h(3h+1)-(-1)^h}{9}.
\end{equation*}
\end{proof}

\begin{remark}
 For $h=0,1,2,\ldots$, the sequence $\frac{2^h(3h+1)-(-1)^h}{9}$ is the sequence A045883 in The On-line Encyclopedia of Integer Sequences (OEIS) \cite{Oeis}. Hence, a new description for the sequence is that, it represents the distance spectral radius of the graph $MC(2^h)$ for $h=0,1,2,\ldots$ where $MC(2^0)$ refers to the trivial graph.  
\end{remark}

We next consider the distance spectral radius of the graph $\Gamma_{3^h}$. The relationship of the two distance spectral radii $\rho(\Gamma_{3^h})$ and $\rho(\Gamma_{3^{h-1}})$ is presented in the next result.

\begin{lemma}
The two distance spectral radii $\rho(\Gamma_{3^{h-1}})$ and $\rho(\Gamma_{3^{h}})$ are related by the equation
\begin{equation*}
    \rho(\Gamma_{3^{h}})=3\rho(\Gamma_{3^{h-1}})+2(3^{h-1}).
\end{equation*}
\label{lemspec3}
\end{lemma}

\begin{proof}
It follows from the proposed method of constructing $bfs(\Gamma_{3^h})$ from $bfs(\Gamma_{3^{h-1}})$ that
\begin{align*}
     \rho(\Gamma_{3^{h}})&=2\Bigg( \rho(\Gamma_{3^{h-1}})+|R[bfs(\Gamma_{3^{h-1}})]|+1+\frac{\rho(\Gamma_{3^{h-1}})}{2}+|L[bfs(\Gamma_{3^{h-1}})]|\Bigg)\\
     &=2\Bigg(\frac{3}{2}\rho(\Gamma_{3^{h-1}})+3^{h-1}-1+1\Bigg)\\
     &=3\rho(\Gamma_{3^{h-1}})+2(3^{h-1}).
\end{align*}
\end{proof}

An explicit formula in determining the distance spectral radius of the graph $\Gamma_{3^h}$ for any positive integer $h$ is presented in the next theorem.

\begin{theorem}
For all positive integer $h$, we have
\begin{equation*}
\rho(\Gamma_{3^h})=2h(3^{h-1}).
\end{equation*}
\label{thspec3}
\end{theorem}

\begin{proof}
For $h=1$, we have $\rho(\Gamma_{3^1})=2=2(1)(3^{1-1})$. Now, let $h>1$ be an integer and suppose that for all $k<h$ we have
$\rho(\Gamma_{3^k})=2k(3^{k-1})$. We show that for $h$, we have $\rho(\Gamma_{3^h})=2h(3^{h-1})$.  

By Lemma \ref{lemspec3} we have
\begin{equation*}
    \rho(\Gamma_{3^{h}})=3\rho(\Gamma_{3^{h-1}})+2(3^{h-1}).
\end{equation*}
Now since $h-1<h$, using our induction hypothesis yields
\begin{align*}
    \rho(\Gamma_{3^{h}})&=3\rho(\Gamma_{3^{h-1}})+2(3^{h-1})\\
    &=3(2(h-1)3^{h-2})+2(3^{h-1})\\
    &=3(2h-2)(3^{h-2})+2(3^{h-1})\\
    &=(2h-2)3^{h-1}+2(3^{h-1})\\
    &=3^{h-1}(2h-2+2)\\
    &=2h(3^{h-1})
\end{align*}

\end{proof}

\begin{remark}
 For $h=1,2,\ldots$, the sequence $2h(3^{h-1})$ is the sequence A212697 in The On-line Encyclopedia of Integer Sequences (OEIS) \cite{Oeis2}. Hence, a new description for the sequence is that, it represents the distance spectral radius of the graph $MC(3^h)$ for $h=1,2,\ldots$.  
\end{remark}

\subsection{Average Distance}

An exact formula for the average distance of the graphs $\Gamma_{2^h}$ and $\Gamma_{3^h}$ are presented in this subsection.

The first result gives the exact formula for the average distance of the graph $\Gamma_{2^h}$.

\begin{theorem}
The average distance of $\Gamma_{2^{h}}$ is
\begin{equation*}
    \mu(\Gamma_{2^h})=
    \begin{cases}
    \frac{2^h(3h+1)-1}{9(2^h-1)}\ \ \ \mbox{if $h$ is even}\\
    \frac{2^h(3h+1)+1}{9(2^h-1)}\ \ \ \mbox{if $h$ is odd}
    \end{cases}
\end{equation*}

\begin{proof}
The proof follows from the definition of average distance, Theorem \ref{thspec2} and the fact that $\Gamma_{2^h}$ is a transmission regular graph.
\end{proof}
\end{theorem}

\begin{remark}
 Earlier authors such as Wong and Coppersmith \cite{Wong} as well as Stojmenovic \cite{Stoj} defined the average distance of a graph as the sum of all the entries in the graph's distance matrix divided by the number of entries in the distance matrix. Stojmenovic (Theorem 10 Stojmenovic, \cite{Stoj}) determined the average distance of $\Gamma_{2^h}$ to be
 \begin{equation}
     \mu(\Gamma_{2^h})=\frac{h}{3}+\frac{1}{9}+\frac{(-1)^h}{9(2^{h-1})}-\frac{(-1)^h}{3(2^h)}.
     \label{ave1}
 \end{equation}
 The proof of (Theorem 10 Stojmenovic, \cite{Stoj}) depends on a  summation formula (Theorem 9, Stojmenovic \cite{Stoj}) as a result of analysis on finding the average distance of the graph $\Gamma_{m^h}$ where $m$ is any even positive integer. Prior to stating (Theorem 10 Stojmenovic, \cite{Stoj}), Stojmenovic first stated that they will find the exact value of the summation for base 2 and that it is possible to follow the same approach for any even base $m$ but their analysis did not lead to a clear and concise formula.
 
 We give a simpler proof for (Theorem 10 Stojmenovic, \cite{Stoj}) using the results in this paper. Note that using Theorem 22 and the fact that $\Gamma_{2^h}$ is a transmission regular we have
 \begin{equation}
    2^h\times \frac{2^h(3h+1)-(-1)^h}{9}
    \label{ave2}
 \end{equation}
 as the sum of all the matrix entries in $D(\Gamma_{2^h})$. Dividing expression \eqref{ave2} by the number of entries in $D(\Gamma_{2^h})$ which is $2^h\times 2^h$ we get
 \begin{equation}
     \frac{2^h(3h+1)-(-1)^h}{9(2^h)}.
     \label{ave3}
 \end{equation}
 Now, simplifying the expression in equation \eqref{ave1} gives expression \ref{ave3}. This proves (Theorem 10 Stojmenovic, \cite{Stoj}).  
\end{remark}

We next give the exact formula for the average distance of the graph $\Gamma_{3^h}$.

\begin{theorem}
The average distance of $\Gamma_{3^{h}}$ is
\begin{equation*}
    \mu(\Gamma_{3^h})=\frac{(2h)(3^{h-1})}{3^h-1}.
\end{equation*}

\begin{proof}
The proof follows from the definition of average distance, Theorem \ref{thspec3} and the fact that $\Gamma_{3^h}$ is a transmission regular graph.
\end{proof}
\end{theorem}

\begin{remark}
 Using the old definition of average distance, Theorem \ref{thspec3} and the fact that $\Gamma_{3^h}$ is a transmission regular graph we will have a simpler proof for (Theorem 7 Stojmenovic, \cite{Stoj}) for base 3. Theorem 7, \cite{Stoj} was originally proven by Wong and Coppersmith \cite{Wong}.  
\end{remark}

\subsection{Vertex-forwarding Index and Bounds for Edge-forwarding Index} 

In this final subsection of section 3, we present the results for the exact value of the vertex-forwarding index and the bounds for edge-forwarding index of the graphs $\Gamma_{2^h}$ and $\Gamma_{3^h}$.

The first two results about the exact value of the vertex-forwarding index of the graphs $\Gamma_{2^h}$ and $\Gamma_{3^h}$ follows immediately from Lemma \ref{lxi}, Theorem \ref{thspec2} and Theorem \ref{thspec3}.  

\begin{theorem}
The vertex-forwarding index of the graph $\Gamma_{2^h}$ is given by
\begin{equation*}
\xi(\Gamma_{2^h})=\frac{(3h+1)(2^h)-(-1)^h}{9}-(2^h-1).
\end{equation*}
\end{theorem} 

\begin{theorem}
The vertex-forwarding index of the graph $\Gamma_{3^h}$ is given by
\begin{equation*}
\xi(\Gamma_{3^h})=(3^{h-1})(2^h-3)+1.
\end{equation*}
\end{theorem}

The final two results of this paper gives the upper and lower bounds for the edge-forwarding index of the graphs $\Gamma_{2^h}$ and $\Gamma_{3^h}$. The proof of the results follows immediately from Lemmas \ref{lemvreg} and \ref{lpi} and Theorems \ref{thspec2} and \ref{thspec3}. 

\begin{theorem}
The edge-forwarding index of the graph $\Gamma_{2^h}$ is given bounded by
\begin{equation*}
\frac{(3h+1)(2^{h+1})-2(-1)^h}{18h-9}\leq \pi(\Gamma_{2^h})\leq \frac{2^h+(3h+1)(2^h)-(-1)^h-36h+27}{9}.
\end{equation*}
\end{theorem}

\begin{theorem}
The edge-forwarding index of the graph $\Gamma_{3^h}$ is given bounded by
\begin{equation*}
2(3^{h-1})\leq \pi(\Gamma_{3^h})\leq 3^{h-1}(3+2h)-4h+1.
\end{equation*}
\end{theorem}

\section{Conclusion and Future Work}
In this paper, we successfully determined the diameter, distance spectral radius, average distance, vertex-forwarding index and a bound for the edge-forwarding index of the multiplicative circulant graph of order $2^h$ and $3^h$. The success in determining those graph parameters is due to the proposed method for determining the distance of the 0-vertex to any other vertices in each of the graphs. The method uses a combination of bfs method and the recursive properties of the graphs. Prior to the creation of this paper, we obtained similar graph parameters for the graph $MC({\bf m}^h)$ where $m>3$ is odd. This leaves obtaining the same graph parameters for $MC(m^h)$ where $m>2$ is even an open problem. The result of this paper also enable us to compute the exact values of some distance-based topological indices such as Wiener, hyper-Wiener, Schultz, Harary, additively weighted Harary and multiplicatively weighted Harary index for mulltiplicative circulant networks of order power of two and three. Hence, one might also consider to study other topological indices and distance coloring for the graph $MC(m^h)$.      

\section{Acknowledgements}

The authors are thankful to the DOST ASTHRDP-NSC for the funding of this research. 

\end{document}